\documentclass[reqno]{amsart}
\usepackage{hyperref}

\begin{document}
\title [A stochastic ratio-dependent predator-prey model]
{Dynamics of a stochastic ratio-dependent predator-prey model}

\author[Nguyen Thi Hoai Linh, Ta Viet Ton]
{Nguyen Thi Hoai Linh, Ta Viet Ton$^*$}

\address{Nguyen Thi Hoai Linh \newline
Department of Information and Physical Sciences\\
Osaka University, Osaka 565-0871}

\address{Ta Viet Ton \newline
Department of Applied Physics, Osaka University, Suita, Osaka 565-0871, Japan}
\email{taviet.ton[at]ap.eng.osaka-u.ac.jp}

\thanks{{*} Corresponding author}
\keywords{Predator-prey model; Ratio-dependent functional 
\hfill\break\indent response; Stochastic differential equation; Moment boundedness;
Asymptotic 
\hfill\break\indent  behaviour}

\begin{abstract}
In this paper, we consider a stochastic ratio-dependent predator-prey model. We firstly prove the existence, uniqueness and positivity of the solutions. Then, the boundedness of moments of population are studied. Finally, we show the upper-growth rates and exponential death rates of population under some conditions.
\end{abstract}

\maketitle
\numberwithin{equation}{section}
\newtheorem{theorem}{Theorem}[section]
\newtheorem{lemma}[theorem]{Lemma}
\newtheorem{definition}[theorem]{Definition}
\newtheorem{remark}[theorem]{Remark}
\newtheorem{proposition}[theorem]{Proposition}
\newtheorem{hypothesis}[theorem]{Hypothesis}
\newtheorem{corollary}[theorem]{Corollary}
\allowdisplaybreaks

\section {Introduction}
The stable predator-prey Lotka-Volterra equation
\begin{equation}\label{E1}
\begin{cases}
\dot x_t= (a_1 -c_1 y_t-b_1 x_t)x_t, \\
\dot y_t=(-a_2 +c_2 x_t-b_2 y_t)y_t,
\end{cases}
\end{equation}
has attracted much interest. In this model, $x_t$ and $y_t$ represent the population density of prey species and predator species at time $t$, respectively; $a_1$ is the intrinsic growth rate of prey in the absence of predator,  $a_2$ is the death rate of predator in the absence of prey; $b_i$ measures the inhibiting effect of environment on two species; $c_i (i=1,2)$ are coefficients of the effect of a species on the other and those coefficients  are positive constants. It is well-known that the solution of \eqref{E1} is asymptotically stable.  

For the stochastic predator-prey Lotka-Volterra equation,  we have to mention one of the first attempts in this direction, a very interesting paper of Arnold et al. \cite{Ar} where the authors used the theory of Brownian
motion and the related white noise models to study the sample paths of the equation
\begin{equation}\label{E2}
\begin{cases}
dx_t= (a_1 -c_1 y_t-b_1 x_t)x_tdt+\sigma x_t dw_t,\\
dy_t=(-a_2 +c_2 x_t-b_2 y_t)y_t dt+\rho y_t dw_t,
\end{cases}
\end{equation}
where $\{w_t, t\geq 0\} $ is a $1$\,{-}\,dimensional Brownian motion defined on a  complete probability space with filtration $(\Omega, \mathcal F,\{\mathcal F_t\}_{t\geq 0},\mathbb P)$ 
 satisfying the usual conditions (see \cite{IS}).  The positive numbers $ \sigma, \rho,$  respectively, are the coefficients of the effect of environmental stochastic perturbation on the prey and on the predator population. In this model, the random factor makes influences on the intrinsic growth rates of prey and predator.

For both models \eqref{E1} and \eqref{E2}, the predator consumes the prey with functional response of type $c_1 x_t y_t$ and contributes to its growth rate $c_2 x_t y_t.$ For such models, we refer to \cite{MMR, MSR}. However, some biologists have argued that in many situations, especially when predators have to search for food, the functional response should depend on both prey's and predator's densities; for example, a Holling's type II functional response in a type of $\frac{x_ty_t}{\alpha+ \beta x_t}$ \cite{SG,TVT}, a ratio-dependent one in a type of  $\frac{x_ty_t}{x_t+\gamma y_t}$ \cite{AG}. The model in those papers are deterministic.

In this paper, we consider a ratio-dependent predator-prey model in the random environment in which the white noise makes the effect on both the growth rates  of species 
 and the inhibiting effects of environment on two species, i.e., $a_1 \to a_1 + \sigma_1  \dot w_t, -a_2 \to - a_2 +\rho_1  \overset{.} {w}_t,$ $-b_1 \to -b_1 + \sigma_1  \dot w_t, -b_2 \to -b_2 + \rho_2  \dot w_t.$
 The model is as follows
\begin{equation}\label{E4}
\begin{cases}
\begin{aligned}
dx_t=& \left[a_1(t)-b_1 (t) x_t-\frac{c_1(t) y_t}{x_t+e(t) y_t}\right] x_tdt\\
&+[\sigma_1 (t)+\sigma_2 (t) x_t ]x_tdw_t,\\
dy_t=&\left[-a_2(t)- b_2 (t) y_t +\frac{c_2(t) x_t}{x_t+e(t) y_t} \right]y_tdt\\
&+[\rho_1 (t)+\rho_2 (t) y_t ]y_t dw_t.
\end{aligned}
\end{cases}
\end{equation}
The white noise here is effected by one-dimensional Brownian motion  $\{w_t, t\geq 0 \}$ defined on the complete probability space  $(\Omega, \mathcal F, \mathcal F_t, \mathbb P)$ with filtration $\{\mathcal F_t\}_{t\geq 0}$ satisfying the usual conditions. Throughout this paper, we suppose that all coefficients of deterministic part of \eqref{E4} are bounded by some positive constants, i.e., their infimum and supremum on $t\geq 0$ are positive, and assume that $\sigma_i, \rho_i \, (i=1,2)$ are nonnegative bounded functions.

The first aim of this paper is to study the existence, uniqueness and positivity of solutions of system \eqref{E4} by using some Lyapunov functions. The next aim is investigating  the boundedness of moments. And the final one is to show some upper-growth rates as well as showing the exponential death rates of species under some conditions. 

The paper is organized as follows. In Section 2, we prove the  existence, uniqueness and  positivity of solutions. Section 3 produces asymptotic estimations of moments. In section 4,  the upper-growth rates and exponential death rates of species are studied.
\section{Global solution}
In this section, we show that the solution of system \eqref{E4} is positive and global. Denote $\mathbb R^n_+=\{(x_1,\cdots,x_n) \in \mathbb R^n\,{:}\, x_i>0 \,\,(i\geq1)\}, \,\, \mathbb R^n_{+0}=\{(x_1,\cdots,x_n) \in \mathbb R^n\,{:}\, x_i\geq 0\,\,(i\geq1)\}\,\, (n=1,2).$ Let $g(t)$ be a function, for a
brevity, instead of writing $g(t)$ we write $g$. 
If $g$ is bounded continuous function on $\mathbb R_{+0}$, we denote 
$$g^u= \sup_{t \in \mathbb R_{+0}}\ g(t), \ g^l=\inf_{t \in \mathbb R_{+0}} g(t).$$
 We have the following theorem.
\begin{theorem} \label{T1}
For any given initial value $(x_0,y_0) \in \mathbb R^2_{+},$ there is a unique solution $(x_t, y_t)$ to \eqref{E4} for $t\geq 0$. Further, with probability one, $\mathbb R^2_+$ is positively invariant for \eqref{E4}, i.e.,   $(x_t, y_t)\in \mathbb R^2_+$ a.s. for all $t\geq 0,$ if $(x_0, y_0)\in \mathbb R^2_+$.
\end{theorem}
\begin{proof}
Since some coefficients of \eqref{E4} are not locally Lipschitz continuous, we can not say that there is a unique local solution $(x_t, y_t)$ to \eqref{E4}. However, we consider the following system
\begin{equation} \label{E4.1.1}
\begin{cases}
\begin{aligned}
d\xi_t= &\left[a_1-\frac{1}{2}(\sigma_1+\sigma_2 \exp\{\xi_t\})^2-b_1 \exp\{\xi_t\}-\frac{c_1\exp\{\eta_t\}}{\exp\{\xi_t\}+e \exp\{\eta_t\}}\right] dt\\
&+(\sigma_1+ \sigma_2 \exp\{\xi_t\}) dw_t,\\
d\eta_t=&\left[-a_2-\frac{1}{2}(\rho_1+\rho_2 \exp\{\eta_t\})^2-b_2 \exp\{\eta_t\}+\frac{c_2 \exp\{\xi_t\}}{\exp\{\xi_t\}+e \exp\{\eta_t\}} \right]dt\\
&+(\rho_1+\rho_2 \exp\{\eta_t\}) dw_t,
\end{aligned}
\end{cases}
\end{equation}
with an initial value $(\xi_0, \eta_0)=(\ln x_0, \ln y_0).$ Since the coefficients of \eqref{E4.1.1} are locally Lipschitz continuous, there is a unique local solution $(\xi_t,\eta_t)$ to \eqref{E4.1.1} for $t\in [0,\tau),$ where $\tau$ is the explosion time (Arnold \cite{A} or Friedman \cite{F}). Therefore, by It\^o's formula, $(x_t, y_t)=( \exp\{\xi_t\},  \exp\{\eta_t\})$ is the unique positive local solution to \eqref{E4}  for  $t\in [0,\tau)$  with the initial value $(x_0, y_0).$ To show the solution is global, we need to show that $\tau=\infty$ a.s.  We use the technique of localization dealt with in \cite{IS,M}. Let $k_0>0$ be sufficiently large for $x_0$ and $y_0$ lying within the interval $[\frac{1}{k_0},k_0]$. Let us define a sequence of stopping times \cite [Problem 2.7, p.7]{IS} for each integer $k\geq k_0$ by
$$\tau_k=\inf \left\{t\geq 0 \,{:}\, x_t\notin (\frac{1}{k},k) \text{ or } y_t\notin (\frac{1}{k},k)\right\}$$
(with the convention $\inf\emptyset=\infty$).  Since $\tau_k$ is nondecreasing as $k \rightarrow \infty$, there exists the limit $\tau_{\infty}=\lim_{k \rightarrow \infty} \tau_k$. Then $\tau_{\infty}\leq \tau$ a.s. Now, we will show that $\tau_{\infty}=\infty$ a.s. If this statement is false, then there exist $T>0$ and $\varepsilon \in (0,1)$ such that $\mathbb P\{\tau_{\infty}\leq T\}>\varepsilon.$ Thus, by denoting $\Omega_k=\{\tau_k \leq T\},$ there exists $k_1\geq k_0$ such that 
\begin{equation} \label{E5}
\mathbb P(\Omega_k)\geq \varepsilon \hspace{1cm}  \text{ for all } k\geq k_1.
\end{equation}
Let $\theta_i\in (0,1)\, (i=1,2).$ We consider the following function 
$$V(x,y)=x^{\theta_1}-\ln x+ y^{\theta_2}-\ln y-\sum_{i=1}^2\frac{1+\ln \theta_i}{\theta_i}\cdot$$
 Because $x^{\theta_1}-\ln x- \frac{1+\ln \theta_1}{\theta_1}\geq 0 $ for all $ x> 0,$ we have $V\in C^2 (\mathbb R^2_+,\mathbb R_{+0}).$ If $(x_t,y_t) \in \mathbb R^2_+,$ by using It\^o's formula, we get
\begin{equation} \label{E6}
dV(x_t,y_t)=f(x_t,y_t,t)dt+g(x_t,y_t,t) dw_t,
\end{equation}
where
\begin{equation} \label{E7}
\begin{aligned}
g(x,y,t)=&(\theta_1 x^{\theta_1}-1)(\sigma_1+\sigma_2 x)+(\theta_2 y^{\theta_2}-1)(\rho_1+\rho_2 y),\\
f(x,y,t)=&(\theta_1 x^{\theta_1}-1)\left[ a_1-b_1x-\frac{c_1 y}{ x+e y}\right]\\
&+(\theta_2 y^{\theta_2}-1)\left[- a_2-b_2y-\frac{c_2 x}{ x+e y}\right]\\
&+\frac{1}{2}[\theta_1(\theta_1-1)x^{\theta_1}+1](\sigma_1+\sigma_2 x)^2\\
&+\frac{1}{2}[\theta_2(\theta_2-1)y^{\theta_2}+1](\rho_1+\rho_2 y)^2.
\end{aligned}
\end{equation}
It is easy to see  from $\theta_i\in (0,1)$ and from  \eqref{E7} that the function $f(x,y,t)$ is bounded above, say by $M$, in $\mathbb R^2_+  \times \mathbb R_{+0}$.  It then follows from  $(x_{t\wedge \tau_k},y_{t\wedge\tau_k})\in \mathbb R^2_+$ and from  \eqref{E6} that
$$\int_0^{T\wedge \tau_k}dV(x_t,y_t)\leq \int_0^{T\wedge \tau_k} M dt+ \int_0^{T\wedge \tau_k} g(x_t,y_t,t) dw_t.$$
Taking expectations yields
\begin{equation} \label{E8}
\mathbb E V(x_{T\wedge \tau_k},y_{T\wedge \tau_k}) \leq V(x_0, y_0) + M \mathbb E (T\wedge \tau_k) \leq V(x_0, y_0)+ MT.
\end{equation}
On the other hand, for every $\omega\in \Omega_k,$ either $x_{\tau_k}(\omega)$  or $y_{\tau_k}(\omega)$ belongs to the set $\{k, \frac{1}{k}\}$. Then
\begin{align*}
V(x_{T\wedge \tau_k}(\omega),y_{T\wedge \tau_k}(\omega)) \geq &\min\Big\{k^{\theta_i}-\ln k- \frac{1+\ln \theta_i}{\theta_i},\\
& \frac{1}{k^{\theta_i}}-\ln \frac{1}{k}- \frac{1+\ln \theta_i}{\theta_i}\quad  (i=1,2)\Big\} \\
=&\min\Big\{k^{\theta_i}-\ln k- \frac{1+\ln \theta_i}{\theta_i},\\
& \frac{1}{k^{\theta_i}}+\ln k- \frac{1+\ln \theta_i}{\theta_i}\quad  (i=1,2)\Big\}.
\end{align*}
We therefore get from \eqref{E5} that
\begin{align*}
\mathbb E V(x_{T\wedge \tau_k},y_{T\wedge\tau_k})\geq &\mathbb E [1_{\Omega_k}V(x_{T\wedge \tau_k},y_{T\wedge \tau_k})]\\
\geq &\varepsilon \min\Big\{k^{\theta_i}-\ln k- \frac{1+\ln \theta_i}{\theta_i},\\
& \frac{1}{k^{\theta_i}}+\ln k- \frac{1+\ln \theta_i}{\theta_i}\quad  (i=1,2)\Big\}.
\end{align*}
It then follows from \eqref{E8} that 
$$V(x_0, y_0)+ MT\geq \varepsilon \min \Big\{k^{\theta_i}-\ln k- \frac{1+\ln \theta_i}{\theta_i}, \frac{1}{k^{\theta_i}}+\ln k- \frac{1+\ln \theta_i}{\theta_i} \quad  (i=1,2)\Big\}.$$
Letting $k \rightarrow \infty$ leads to  $\infty > V(x_0, y_0)+ MT =\infty.$ This is a contradiction. Therefore 
$\tau_{\infty}=\infty$ a.s. Then $\tau=\infty$ a.s., and $(x_t, y_t) \in \mathbb R^2_+$ a.s. The proof is complete.
\end{proof}
\section{Boundedness of moments}
In this section, we shall show the boundedness of species's quantity moments.   In order to capture them, we shall distinguish between the following hypotheses of random perturbation.

(H1) The random factor make only effects on the growth rate of population, i.e., $\sigma_2(t)=\rho_2 (t)=0$ for all $t\geq 0$ and $\sigma_1^l,\rho_1^l>0;$ 

(H2) The random factor make  effects on both the growth rate of population and the inhibiting effects of environment, i.e., $\sigma_i^l, \rho_i^l>0\,\, (i=1,2).$\\
Consider 
\begin{equation*}
\begin{aligned}
LV(x,y)=&\frac{1}{2} [\sigma_1 (t)+\sigma_2 (t)x]^2 x^2\frac{\partial^2 V}{\partial x^2}+\frac{1}{2}  [\rho_1 (t)+\rho_2 (t)y]^2 y^2\frac{\partial^2 V}{\partial y^2}\\
&+ [\sigma_1 (t)+\sigma_2 (t)x][\rho_1 (t)+\rho_2 (t)y] xy  \frac{\partial^2 V}{\partial x \partial y}\\
& +f_1(x,y,t) \frac{\partial V}{\partial x}+f_2(x,y,t) \frac{\partial V}{\partial y},
\end{aligned}
\end{equation*}
the infinitesimal operator of \eqref{E4}, defined on the space $C^2(\mathbb R^2_+,\mathbb R),$ where
$$f_1(x,y,t)=a_1x-\frac{c_1 xy}{x+e y}-b_1 x^2,$$ 
$$ f_2(x,y,t)=-a_2 y+\frac{c_2 xy}{ x+e y} - b_2 y^2,$$
and for any positive numbers  $\theta_1, \theta_2$, we put
\begin{align} \label{E9}
d_2&=\min\{ \theta_ib_i^l\,\,\,  (i=1,2)\}, \quad \theta=\frac{1}{\theta_1+\theta_2}, \notag \\
d_1&=\sup_{t\geq 0}\Big\{\frac{1}{2}\sigma^2_1(t)\theta_1(\theta_1-1)+\frac{1}{2}\rho^2_1(t)\theta_2(\theta_2-1)+\sigma_1(t)\rho_1(t)\theta_1\theta_2+\theta_1 a_1(t)\notag\\
&+[c_2(t)- a_2(t)]\theta_2\Big\},\\
\lambda_1&=\frac{d_1}{d_2}-(\theta_1+\theta_2)\{1-\ln (\theta_1+\theta_2)\}, \notag\\
\lambda_2&(x_0,y_0)=(\theta_1 \ln x_0+\theta_2 \ln y_0)+(\theta_1+\theta_2)\{1-\ln(\theta_1+\theta_2)\}-\frac{d_1}{d_2}\cdot \notag
\end{align}
We have the following theorems.
\begin{theorem} \label {T2}
Under condition {\rm(H1)}, for any  positive  $\theta_1, \theta_2$ and  initial value $(x_0,y_0)\in \mathbb R^2_+,$ solution  of \eqref{E4} satisfies
$$\mathbb E (x^{\theta_1}_ty^{\theta_2}_t) \leq \exp\{\lambda_1+\lambda_2(x_0,y_0) \exp\{-d_2t\}\}\hspace{1cm} \text{ for all } t\geq 0.$$
Consequently, $\limsup_{t\rightarrow \infty} \mathbb E (x^{\theta_1}_ty^{\theta_2}_t)\leq \exp\{\lambda_1\}. $
\end{theorem}
\begin{proof}
Firstly, we prove that
\begin{equation} \label{E10.0}
\mathbb E (x^{\theta_1}_ty^{\theta_2}_t)< \infty \hspace{1cm}\text{  for all } t\geq 0  \text{ and  } \theta_i>0.
\end{equation}
Define a function $V\in C^2 (\mathbb R^2_+,\mathbb R_+)$ by $V(x,y)=x^{\theta_1}y^{\theta_2}.$ For any $t\geq 0$, using It\^o's formula gives that
\begin{equation} \label{E10}
dV(x_t,y_t)=LV(x_t,y_t)dt+(\theta_1\sigma_1+\theta_2\rho_1)V(x_t,y_t)dw_t.
\end{equation}
It is easy to see that
\begin{equation} \label {E11}
\begin{aligned}
LV(x,y)=&\Big[\frac{1}{2}\theta_1(\theta_1-1)\sigma_1^2+\frac{1}{2}\theta_2(\theta_2-1)\rho_1^2+\theta_1\theta_2\sigma_1\rho_1+\theta_1(a_1-b_1x) \\
&-\theta_2(a_2+b_2 y)+\frac{\theta_2c_2 x-\theta_1c_1y}{x+e y}\Big]V(x,y)\\
\leq &[d_1-d_2 (x+y)] V(x,y).
\end{aligned}
\end{equation}
For every integer $k\geq 1$, we define a stopping time $\tau_k=\inf\{t\geq 0 \,{:}\,  x_t+y_t \geq k\}.$ Then, the sequence $\{\tau_k, k\geq 1\}$ is nondecreasing  and  by the positive invariance of $(x_t,y_t)$ on $\mathbb R^2_+$, we have $\lim_{k\rightarrow \infty }\tau_k=\infty$ a.s. It then follows from \eqref{E10} that
\begin{align*}
V(x_{t\wedge \tau_k}, y_{t\wedge \tau_k})=&V(x_0,y_0)+\int_0^{t\wedge\tau_k}LV(x_s,y_s)ds\\
&+\int_0^{t\wedge \tau_k}[\theta_1\sigma_1(s)+\theta_2\rho_1(s)]V(x_s,y_s)dw_s.
\end{align*}
Taking expectations of both sides and using \eqref{E11}, we have
\begin{align*}
\mathbb E V(x_{t\wedge \tau_k}, y_{t\wedge \tau_k})& \leq V(x_0,y_0)+d_1 \mathbb E \int_0^{t\wedge \tau_k}V(x_s,y_s)ds\\
& \leq V(x_0,y_0)+d_1  \int_0^t\mathbb E V(x_{s\wedge \tau_k},y_{s\wedge \tau_k} )ds.
\end{align*}
Thus, by using  Gronwall's inequality, we obtain $$\mathbb E V(x_{t\wedge\tau_k}, y_{t\wedge \tau_k})\leq V(x_0,y_0) \exp\{d_1t\}.$$
Letting $k\rightarrow \infty$ in the latter inequality, we yield that, for all $t\geq 0,$ $\mathbb E V(x_t,y_t)\leq V(x_0,y_0) \exp\{d_1t\},$  from which we deduce \eqref{E10.0}.

Next, since $V(x,y)=x^{\theta_1}y^{\theta_2}\leq (x+y)^{\theta_1+\theta_2}, $ we have  $x+y \geq V^{\theta}(x,y).$ It then follows from \eqref{E11} that 
\begin{equation} \label{E12}
LV(x,y)\leq [d_1-d_2 V^{\theta}(x,y)] V(x,y).
\end{equation}
Applying  \eqref{E10.0} to $(\theta_1(1+\theta), \theta_2(1+\theta))$, we have 
$$\mathbb E \left[V^{1+\theta}(x_t,y_t)\right]=\mathbb E \left[x^{\theta_1(1+\theta)}_ty^{\theta_2(1+\theta)}_t\right]<\infty \hspace{1cm}  \text{ for all } t\geq 0.$$
Then, by using H\"older's inequality,  yields
$$\left[\mathbb E V(x_t,y_t)\right]^{1+\theta}\leq \mathbb E\left[V^{1+\theta}(x_t,y_t)\right].$$
It then follows from \eqref{E10} and from  \eqref{E12} that, for any $t\geq 0$ and $h>0$,
\begin{align} \label{E13}
\mathbb E V(x_{t+h}, y_{t+h})-\mathbb E V(x_t,y_t)&\leq \int_t^{t+h}\left[ d_1\mathbb E V(x_s,y_s)-d_2 \mathbb E V^{1+\theta}(x_s,y_s)\right]ds \notag\\
&\leq  \int_t^{t+h}\left[ d_1\mathbb E V(x_s,y_s)-d_2 \left [\mathbb E V(x_s,y_s)\right]^{1+\theta}\right]ds.
\end{align}
Putting $v(t)=\mathbb E V(x_t,y_t),$ then $0<v(t)<\infty$ for all $t\geq 0$. Further, the continuity of $v(t)$ in $t$ can be seen by the continuity of the solution $(x_t, y_t)$ and the dominated convergence theorem. We define the right upper derivative of $v(t) $ by 
$$D^+v(t)=\limsup_{h \rightarrow 0} \frac{v(t+h)-v(t)}{h}\cdot$$
From \eqref{E13}, we have
$$\frac{v(t+h)-v(t)}{h}\leq \frac{1}{h}\int_t^{t+h}\left[ d_1v(s)-d_2 v^{1+\theta}(s)\right]ds.$$
Letting $h\rightarrow 0$ gives
$D^+v(t)\leq v(t)[d_1-d_2v^{\theta}(t)] \, \, \text{ for all } t\geq 0.$
Therefore,
\begin{align*}
D^+[\exp\{d_2t\} \ln v(t)]&=d_2 \exp\{d_2t\}  \ln v(t)+ \exp\{d_2t\}  \frac{D^+v(t)}{v(t)}\\
&\leq d_2 \exp\{d_2t\}  \ln v(t)+\exp\{d_2t\} [d_1-d_2 v^{\theta}(t)]\\
&=d_1 \exp\{d_2t\}+ d_2 \exp\{d_2t\} [\ln v(t)-v^{\theta}(t)].
\end{align*}
It is easy to see that  $\ln x-x^{\theta}\leq -\frac{1}{\theta}(1+\ln \theta) $ for all $x>0.$ Then
$$D^+[\exp\{d_2t\} \ln v(t)] \leq \left [d_1-\frac{1}{\theta}(1+\ln \theta)  d_2\right]\exp\{d_2t\}.$$
Taking integrations of both sides yields
\begin{align*}
\exp\{d_2t\} \ln v(t)\leq &\ln v(0)+\left [\frac{d_1}{d_2}-\frac{1}{\theta}(1+\ln \theta) \right]\left [\exp\{d_2t\}-1\right]\\
=&\lambda_2 +\lambda_1 \exp\{d_2t\}.
\end{align*}
Consequently, we have $\ln v(t)\leq \lambda_1+\lambda_2 \exp\{-d_2t\},$ from which follows the first statement of theorem.  Letting $t\to \infty$ in the latter inequality, we get the second one. The proof is complete.
\end{proof}
\begin{theorem} \label{T3}
Under condition {\rm(H2)}, for any $\theta_i \in (0,1], \varrho_i \in [0,3)$ and $ \varsigma_i \in \mathbb R_+,$  there exist positive constants $K_1=K_1(\theta_i, \varsigma_i)$ and $K_2=K_2(\varrho_i, \varsigma_i) \,(i=1,2)$ satisfying the following  for any initial value $(x_0, y_0) \in \mathbb R^2_+$
\begin{itemize}
\item [(i)]
$\limsup_{t\to \infty} \mathbb E \left[\varsigma_1x^{\theta_1}_t+\varsigma_2y^{\theta_2}_t\right] \leq K_1;$
\item [(ii)] $\limsup_{t\to \infty} \frac{1}{t} \int_0^t \mathbb E [\varsigma_1x^{\varrho_1}_s+\varsigma_2y^{\varrho_2}_s]ds \leq K_2.$
\end{itemize}
\end{theorem}
\begin{proof}
Consider a function $V_1\,{:}\, \mathbb R^2_+ \to \mathbb R^2_+$ defined by $V_1(x,y)=\varsigma_1x^{\theta_1}+\varsigma_2y^{\theta_2}.$ For any $t\geq 0$, by using It\^o's formula, we have
\begin{equation} \label{E14}
\begin{aligned}
dV_1(x_t,y_t)=&LV_1(x_t,y_t)dt+\Big\{ \theta_1 \varsigma_1[\sigma_1 (t)+\sigma_2 (t) x_t]x^{\theta_1}_t\\
&+\theta_2 \varsigma_2[\rho_1 (t)+\rho_2 (t) y_t]y^{\theta_2}_t\Big\}dw_t,
\end{aligned}
\end{equation}
where
\begin{equation} \label {E15}
\begin{aligned}
LV_1(x,y)=&\frac{1}{2}\theta_1(\theta_1-1) \varsigma_1 [\sigma_1 (t)+\sigma_2 (t) x]^2x^{\theta_1}\\
&+\frac{1}{2}\theta_2(\theta_2-1)  \varsigma_2 [\rho_1 (t)+\rho_2 (t) y]^2y^{\theta_2}\\
&+\theta_1\varsigma_1x^{\theta_1}\Big [a_1(t)-b_1 (t)x-\frac{c_1(t)y}{x+e(t) y}\Big]\\
&+\theta_2\varsigma_2y^{\theta_2}\Big [-a_2(t)+\frac{c_2(t)x}{ x+e(t) y}-b_2 (t)y\Big].\\
\end{aligned}
\end{equation}
Then, from $\theta_i \in (0,1]$ and from $ \varsigma_i \in \mathbb R_+ \, (i=1,2),$  there exists $K_1=K_1(\theta_1, \theta_2, \varsigma_1, \varsigma_2)$ such that $LV_1(x,y)+V_1(x,y)\leq K_1$ for all $(x,y,t)\in \mathbb R^2_+  \times \mathbb R_{+0}.$ Applying It\^o's formula yields
\begin{align} \label{E16}
d[e^tV_1(x_t,y_t)]=&e^t[V_1(x_t,y_t)+LV_1(x_t,y_t)]dt \notag\\
&+e^t[\theta_1 \varsigma_1(\sigma_1+\sigma_2 x_t)x^{\theta_1}_t+\theta_2 \varsigma_2(\rho_1+\rho_2 y_t)y^{\theta_2}_t)]dw_t \notag\\
\leq & K_1 e^t dt+e^t[\theta_1 \varsigma_1(\sigma_1+\sigma_2 x_t)x^{\theta_1}_t+\theta_2 \varsigma_2(\rho_1+\rho_2 y_t)y^{\theta_2}_t)]dw_t. 
\end{align}
Using the sequence of stopping times $\{\tau_k\}_{k=1}^{\infty} $ defined in the proof of Theorem \ref{T1} and from  \eqref{E16}, we have
$$\mathbb E \left[e^{t\wedge \tau_k}V_1(x_{t\wedge \tau_k},y_{t\wedge \tau_k})\right]\leq V_1(x_0,y_0)+ K_1(\mathbb E e^{t\wedge \tau_k}-1) \hspace{1cm}\text{ for all } t\geq 0.$$
Letting $k\rightarrow \infty$ in the latter inequality with a fact that  $V(x_{t\wedge\tau_k}, y_{t\wedge \tau_k})>0$ and $0<e^{t\wedge\tau_k}\leq e^{t} \, a.s.$, and using Fatou's lemma, we obtain
$$e^t\mathbb E V_1(x_t,y_t)\leq V_1(x_0,y_0)+ K_1(e^t-1).$$ Therefore, 
$\limsup_{t\to \infty} \mathbb E V_1(x_t,y_t)\leq K_1.$

To prove Part (ii), we consider a function $V_2(x,y)=\varsigma_1x^{\varrho_1}+\varsigma_2y^{\varrho_2}.$  Since $ \varrho_i \in [0,3),$ there exist $\theta_i \in (0,1)\, (i=1, 2)$ such that $0\leq \varrho_i<2+\theta_i.$
Then, from  \eqref{E15}  there exists $K_2=K_2(\varrho_i, \varsigma_i)$ such that 
$LV_1(x,y)+ V_2(x,y)\leq K_2$  for all $ (x,y,t)\in \mathbb R^2_+  \times \mathbb R_{+0}.$
Using \eqref{E14} gives
\begin{align*}
V_1(x_t,y_t)\leq &V_1(x_0,y_0)+ \int_0^t [K_2-V_2(x_s,y_s)] ds\\
& + \int_0^t [\theta_1 \varsigma_1(\sigma_1+\sigma_2 x_s)x^{\theta_1}_s+\theta_2 \varsigma_2(\rho_1+\rho_2 y_s)y^{\theta_2}_s)]dw_s.
\end{align*}
Taking expectations of both sides, we obtain
$$\mathbb E V_1(x_t,y_t)+\int_0^t \mathbb E V_2(x_s,y_s) ds \leq V_1(x_0,y_0)+ K_2 t,$$
from which follows $\int_0^t \mathbb E V_2(x_s,y_s) ds \leq V_1(x_0,y_0)+ K_2 t.$ Therefore,
$$\limsup_{t\to \infty} \frac{1}{t} \int_0^t \mathbb E [\varsigma_1x^{\varrho_1}_s+\varsigma_2y^{\varrho_2}_s]ds\leq K_2.$$
\end{proof}
\section{Upper growth rate estimation}
In this section, we shall show the upper-growth rates  of population under  the case of the random factor making the effect only on the growth rate of population.
\begin{theorem} \label {T4}
Under condition {\rm(H1)}, for any  $ \theta_i\geq 0$ and any initial value $(x_0, y_0) \in \mathbb R^2_+,$ 
$$\limsup_{t\to \infty} \frac{\ln\left[x_t^{\theta_1}y_t^{\theta_2}\right]}{\ln t}\leq \theta_1+\theta_2   \hspace{1cm} a.s.$$
Furthermore, if $ \theta_i\in [0,1) $ then for any $\varsigma_i >0 \, (i=1,2)$ there exists $K=K( \theta_i, \varsigma_i )$ such that
$$\limsup_{t\to \infty} \frac{1}{t} \int_0^t (\varsigma_1 x^{\theta_1}_s+\varsigma_2 y^{\theta_2}_s)ds \leq K  \hspace{1cm} a.s.$$
\end{theorem} 
\begin{proof}
Firstly, we prove the first inequality. Putting $x_t=\exp\{\xi_t\}, y_t=\exp\{\eta_t\}, \vartheta_1=a_1-\frac{\sigma_1^2}{2}, \vartheta_2=a_2+\frac{\rho_1^2}{2}$  and substituting this transformation into system \eqref{E4}, we obtain
\begin{equation} \label{E20}
\begin{cases}
\begin{aligned}
d\xi_t= &\left[\vartheta_1-b_1 \exp\{\xi_t\}-\frac{c_1\exp\{\eta_t\}}{\exp\{\xi_t\}+e \exp\{\eta_t\}}\right] dt+\sigma_1 dw_t,\\
d\eta_t=&\left[-\vartheta_2-b_2 \exp\{\eta_t\}+\frac{c_2 \exp\{\xi_t\}}{\exp\{\xi_t\}+e \exp\{\eta_t\}} \right]dt+\rho_1 dw_t,
\end{aligned}
\end{cases}
\end{equation}
or equivalently
\begin{equation*} 
\begin{cases}
d\xi_t= \left[ \vartheta_1-b_1 x_t-\frac{c_1y_t}{ x_t+e y_t}\right] dt+\sigma_1 dw_t,\\
d\eta_t=\left[- \vartheta_2 - b_2y_t+\frac{c_2x_t}{ x_t+e y_t}\right]dt+\rho_1dw_t.
\end{cases}
\end{equation*}
Fix $p>0$. Applying It\^o's formula to $\exp\{pt\}\xi_t$ and $\exp\{pt\}\eta_t$, from \eqref{E20}, we have
\begin{align} 
\exp\{pt\}\xi_t=&\xi_0+\int_0^t \exp\{ps\}\left[ \vartheta_1-b_1 x_s-\frac{c_1y_s}{x_s+e y_s}\right] ds \notag \\
 &+\int_0^t p\exp\{ps\} \xi_s ds+\int_0^t \sigma_1  \exp\{ps\} dw_s, \label{E22} \\
\exp\{pt\}\eta_t=&\eta_0+\int_0^t \exp\{ps\}\left[- \vartheta_2- b_2 y_s+\frac{c_2 x_s}{ x_s+e y_s} \right] ds \notag \\
 &+\int_0^t p\exp\{ps\} \eta_s ds+\int_0^t \rho_1 \exp\{ps\} dw_s. \label{E23} 
\end{align}
We set 
$$M_{1t}=\int_0^t\sigma_1\exp\{ps\} dw_s, M_{2t}=\int_0^t\rho_1 \exp\{ps\} dw_s,$$
 then $M_{it} \,(i=1,2)$ are real valued continuous martingales vanishing at $t=0$ with quadratic forms
$$<M_{1},M_{1}>_t=\int_0^t \sigma_1^2 \exp\{2ps\} ds, \quad <M_{2},M_{2}>_t=\int_0^t \rho_1^2 \exp\{2ps\} ds.$$
Let $\varepsilon\in (0,1)$ and $\theta>1$. Using the exponential martingale inequality \cite [Theorem 1.7.4]{M}, for every $k\geq 1$ and $ i=1,2$, we have
$$\mathbb P\left\{\sup_{0\leq t\leq k}\left[M_{it}-\frac{\varepsilon}{2}\exp\{-pk\} <M_{i},M_{i}>_t\right]\geq \frac{\theta \exp\{pk\}}{\varepsilon}\ln k\right\}\leq \frac{1}{k^\theta}\cdot$$
It then follows from Borel-Cantelli lemma that there exists an $\Omega_i\subset \Omega$ with $\mathbb P(\Omega_i)=1$ having the following property. For any $\omega\in \Omega_i,$ there exists $k_i=k_i(\omega)$ such that, for all $k\geq k_i$ and $t\in [0,k]$,
\begin{align*}
M_{1t}&\leq \frac{\varepsilon}{2}\exp\{-pk\} <M_1,M_1>_t +\frac{\theta \exp\{pk\}}{\varepsilon}\ln k\\
&= \frac{\varepsilon }{2}\exp\{-pk\} \int_0^t \sigma_1^2\exp\{2ps\} ds +\frac{\theta \exp\{pk\}}{\varepsilon}\ln k,\\
M_{2t}&\leq \frac{\varepsilon}{2}\exp\{-pk\} <M_2,M_2>_t +\frac{\theta \exp\{pk\}}{\varepsilon}\ln k\\
&= \frac{\varepsilon }{2}\exp\{-pk\} \int_0^t \rho_1^2\exp\{2ps\} ds +\frac{\theta \exp\{pk\}}{\varepsilon}\ln k.
\end{align*}
We therefore have from \eqref{E22} and \eqref{E23} that for any $\omega\in \Omega_1\cap \Omega_2$ and $ t\in [0,k], k\geq k_0(\omega),$ where $k_0(\omega)=k_1(\omega)\wedge k_2(\omega),$
\begin{align} 
\exp\{pt\}\xi_t\leq&\xi_0+\int_0^t p\exp\{ps\} \xi_s ds \notag \\
&+\int_0^t \exp\{ps\}\left[\vartheta_1-b_1 x_s-\frac{c_1y_s}{x_s+e y_s}\right] ds \notag\\
& +\frac{\varepsilon }{2}\exp\{-pk\} \int_0^t \sigma_1^2\exp\{2ps\} ds+\frac{\theta \exp\{pk\}}{\varepsilon}\ln k  \notag \\
=& \xi_0+p\int_0^t \exp\{ps\} \xi_s ds +\frac{\theta \exp\{pk\}}{\varepsilon}\ln k +\int_0^t \exp\{ps\}\Big[a_1\notag \\
&-b_1 x_s+\frac{\varepsilon \sigma_1^2}{2}\exp\{-p(k-s)\}-\frac{c_1y_s}{x_s+e y_s}\Big] ds, \notag\\
\label{E24} \\
\exp\{pt\}\eta_t\leq& \eta_0+\int_0^t p\exp\{ps\} \eta_s ds,\notag\\
&+\int_0^t \exp\{ps\}\left[-\vartheta_2-b_2 y_s-\frac{c_2x_s}{x_s+e y_s}\right] ds \notag \\
 &+\frac{\varepsilon }{2}\exp\{-pk\} \int_0^t \rho_1^2\exp\{2ps\} ds+\frac{\theta \exp\{pk\}}{\varepsilon}\ln k\notag\\
 =&  \eta_0+p\int_0^t \exp\{ps\} \eta_s ds +\frac{\theta \exp\{pk\}}{\varepsilon}\ln k +\int_0^t \exp\{ps\}\Big[-\vartheta_2\notag \\
&-b_2 y_s+\frac{\varepsilon\rho_1^2}{2}\exp\{-p(k-s)\}+\frac{c_2x_s}{x_s+e y_s}\Big] ds. \notag\\
  \label{E25} 
\end{align}
From \eqref{E24} and \eqref{E25}, for any $\omega\in \Omega_1\cap \Omega_2$ and $ t\in [0,k], k\geq k_0(\omega),$ we have
\begin{align}  \label{E26}
\exp\{pt\}(\theta_1 \xi_t+\theta_2 \eta_t) \leq &(\theta_1 \xi_0+\theta_2 \eta_0)+\frac{\theta(\theta_1+\theta_2) \exp\{pk\}}{\varepsilon}\ln k  \notag\\
&+\int_0^t \exp\{ps\}\Big[p(\theta_1 \xi_s+\theta_2 \eta_s)+\theta_1\vartheta_1-\theta_2\vartheta_2\notag \\
&+\frac{\varepsilon \left[\theta_1\sigma_1^2+\theta_2\rho_1^2\right] }{2}\exp\{-p(k-s)\} \notag\\
&-b_1\theta_1 x_s - b_2 \theta_2 y_s+\frac{\theta_2c_2x_s-\theta_1c_1y_s}{x_s+e y_s}\Big] ds.
\end{align}
Since $\theta_1, \theta_2 \in \mathbb R_+,$ there exists $H=H(p,\theta_1,\theta_2)>0$ such that for any $ (x,y,t)\in \mathbb R^2_+  \times \mathbb R_{+0},$
$$\Big[p(\theta_1 \ln x+\theta_2 \ln y)+\theta_1\vartheta_1-\theta_2\vartheta_2
-b_1\theta_1 x- b_2 \theta_2 y+\frac{\theta_2c_2x-\theta_1c_1y}{ x+e y}\Big]\leq H.$$
It then follows from \eqref{E26} that for any $\omega\in \Omega_1\cap \Omega_2$ and $ t\in [0,k], k\geq k_0(\omega),$ 
\begin{align*}
\exp\{pt\}(\theta_1 \xi_t+\theta_2 \eta_t) \leq &(\theta_1 \xi_0+\theta_2 \eta_0)+\frac{\theta(\theta_1+\theta_2) \exp\{pk\}}{\varepsilon}\ln k  \notag\\
&+\int_0^t \exp\{ps\}\Big[H+\frac{\varepsilon (\theta_1\sigma_1^2+\theta_2\rho_1^2) }{2}\Big] ds\notag\\
\leq&(\theta_1 \xi_0+\theta_2 \eta_0)+\frac{\theta(\theta_1+\theta_2) \exp\{pk\}}{\varepsilon}\ln k  \notag\\
&+\frac{1}{p} \left[H+\frac{\varepsilon (\theta_1{\sigma_1^u}^2+\theta_2{\rho_1^u}^2) }{2}\right](\exp\{pt\}-1).
\end{align*}
Thus,
\begin{align}  \label{E27}
\theta_1 \xi_t+\theta_2 \eta_t \leq &(\theta_1 \xi_0+\theta_2 \eta_0)\exp\{-pt\}+\frac{\theta(\theta_1+\theta_2) \exp\{p(k-t)\}}{\varepsilon}\ln k  \notag\\
&+\frac{1}{p} \left[H+\frac{\varepsilon (\theta_1{\sigma_1^u}^2+\theta_2{\rho_1^u}^2) }{2}\right](1-\exp\{-pt\})\notag\\
\leq &(\theta_1 \xi_0+\theta_2 \eta_0)+\frac{\theta(\theta_1+\theta_2) \exp\{p(k-t)\}}{\varepsilon}\ln k  \notag\\
&+\frac{1}{p} \left[H+\frac{\varepsilon (\theta_1{\sigma_1^u}^2+\theta_2{\rho_1^u}^2) }{2}\right],\notag\\
\end{align}
For any $\omega\in \Omega_1\cap \Omega_2$, $k\geq k_0(\omega)$ and $t\in [k-1,k]$, from \eqref{E27} we have
\begin{align*}
\frac{\theta_1 \xi_t+\theta_2 \eta_t}{\ln t}\leq &\frac{1}{\ln (k-1)}\Big[(\theta_1 \xi_0+\theta_2 \eta_0)+\frac{\theta(\theta_1+\theta_2) \exp\{p\}}{\varepsilon}\ln k \notag\\
&+\frac{1}{p} \Big\{H+\frac{\varepsilon (\theta_1{\sigma_1^u}^2+\theta_2{\rho_1^u}^2) }{2}\Big\}\Big],\notag
\end{align*}
from which implies
$$\limsup_{t\to\infty} \frac{\theta_1 \xi_t+\theta_2 \eta_t}{\ln t}\leq \frac{\theta(\theta_1+\theta_2) \exp\{p\}}{\varepsilon}\cdot$$
Letting $\varepsilon\to 1^-, \theta \to 1^+, p\to 0^+$ and noting  $\mathbb P(\Omega_1\cap \Omega_2)=1$  yields
$$\limsup_{t\to\infty} \frac{\theta_1 \xi_t+\theta_2 \eta_t}{\ln t}\leq \theta_1+\theta_2   \hspace{1cm} \text { a.s.,}$$
i.e.,
$$\limsup_{t\to \infty} \frac{\ln\left[x_t^{\theta_1}y_t^{\theta_2}\right]}{\ln t}\leq \theta_1+\theta_2  \hspace{1cm} \text { a.s. }$$
Now, we prove the remain inequality. Putting $V(x,y)=\ln (\varsigma_1 x^{\theta_1}+\varsigma_2 y^{\theta_2})$
 and using It\^o's formula, we get easily that
\begin{align} \label{E17}
dV(x_t,y_t)=& \Big[  \frac{\varsigma_1 \theta_1 x^{\theta_1} }{\varsigma_1 x^{\theta_1}+\varsigma_2 y^{\theta_2}}(a_1-b_1 x-\frac{c_1 y}{x+e y}) \notag\\
&  + \frac{\varsigma_2 \theta_2 y^{\theta_2} }{\varsigma_1 x^{\theta_1}+\varsigma_2 y^{\theta_2}}(-a_2-b_2 y+\frac{c_2x}{x+e y}) \notag  \\
&+ \frac{[\varsigma_1 \theta_1 (\theta_1-1)x^{\theta_1}(\varsigma_1 x^{\theta_1}+\varsigma_2 y^{\theta_2})-\varsigma_1^2 \theta_1^2 x^{2\theta_1} ] \sigma_1}{2(\varsigma_1 x^{\theta_1}+\varsigma_2 y^{\theta_2})^2}\notag  \\
 &+ \frac{[\varsigma_2 \theta_2 (\theta_2-1)y^{\theta_2}(\varsigma_1 x^{\theta_1}+\varsigma_2 y^{\theta_2})-\varsigma_2^2 \theta_2^2 y^{2\theta_2} ] \rho_1}{2(\varsigma_1 x^{\theta_1}+\varsigma_2 y^{\theta_2})^2} \notag \\ 
 &-\frac{\varsigma_1\varsigma_2 \theta_1\theta_2\sigma_1 \rho_1x^{\theta_1}y^{\theta_2}  }{2(\varsigma_1 x^{\theta_1}+\varsigma_2 y^{\theta_2})^2}  \Big]dt\notag\\
 &+\frac{\varsigma_1\theta_1 \sigma_1x^{\theta_1}+\varsigma_2\theta_2 \rho_1 y^{\theta_2}}{\varsigma_1 x^{\theta_1}+\varsigma_2 y^{\theta_2}} dw_t\notag\\
 =&P(x,y,t)dt+\frac{\varsigma_1\theta_1 \sigma_1 x^{\theta_1}+\varsigma_2\theta_2 \rho_1 y^{\theta_2}}{\varsigma_1 x^{\theta_1}+\varsigma_2 y^{\theta_2}} dw_t,
\end{align}
where 
\begin{align*}
P(x,y,t)=&\frac{\varsigma_1 \theta_1 x^{\theta_1} }{\varsigma_1 x^{\theta_1}+\varsigma_2 y^{\theta_2}}(a_1-b_1 x-\frac{c_1 y}{x+e y})\\
&  + \frac{\varsigma_2 \theta_2 y^{\theta_2} }{\varsigma_1 x^{\theta_1}+\varsigma_2 y^{\theta_2}}(-a_2-b_2 y+\frac{c_2x}{x+e y})   \\
&+ \frac{\varsigma_1 \theta_1 \sigma_1[\varsigma_2(\theta_1-1) y^{\theta_2}-\varsigma_1 x^{\theta_1} ] x^{\theta_1}}{2(\varsigma_1 x^{\theta_1}+\varsigma_2 y^{\theta_2})^2}  \\
 &+ \frac{\varsigma_2 \theta_2 \rho_1[\varsigma_1(\theta_2-1) x^{\theta_1}-\varsigma_2 y^{\theta_2} ] y^{\theta_2}}{2(\varsigma_1 x^{\theta_1}+\varsigma_2 y^{\theta_2})^2}\\ 
 &-\frac{\varsigma_1\varsigma_2 \theta_1\theta_2 \sigma_1 \rho_1x^{\theta_1}y^{\theta_2}  }{2(\varsigma_1 x^{\theta_1}+\varsigma_2 y^{\theta_2})^2} \cdot
\end{align*}
Putting 
$$K=\sup_{(x,y,t)\in R^2_+\times \mathbb R_{+0}}  \left [P(x,y,t)+(\varsigma_1 x^{\theta_1}+\varsigma_2 y^{\theta_2})\right]$$ 
and
$$M_t=\int_0^t \frac{\varsigma_1\theta_1 \sigma_1x^{\theta_1}+\varsigma_2\theta_2 \rho_1y^{\theta_2}}{\varsigma_1 x^{\theta_1}+\varsigma_2 y^{\theta_2}} dw_s,$$
then $\{M_t, \mathcal F_t, t\geq 0\}$ is a martingale, and by $ \theta_i\in [0,1), $  we see that  $K<\infty.$ The quadratic variation of $M_t$  can be shown by using \cite[Theorem 5.14, p.25]{M} as follows.
$$<M,M>_t=\int_0^t \left[\frac{\varsigma_1\theta_1\sigma_1 x^{\theta_1}+\varsigma_2\theta_2 \rho_1y^{\theta_2}}{\varsigma_1 x^{\theta_1}+\varsigma_2 y^{\theta_2}}  \right]^2ds.$$
It is easy to see that 
$$\limsup_{t\to \infty}\frac{<M,M>_t}{t}\leq \left[\max\{\theta_1 \sigma_1^u,\theta_2 \sigma_2^u\}\right]^2.$$
So, using the strong law of large numbers for martingale  \cite[Theorem 1.3.4]{M}, we have
\begin{equation} \label{E19}
\lim_{t\to \infty} \frac{M_t}{t}=0  \hspace{1cm}  \text{ a.s.}
\end{equation}
On the other hand,  from \eqref{E17}, we have
$$0<V(x_t,y_t)\leq \int_0^t [K -(\varsigma_1 x^{\theta_1}_s+\varsigma_2 y^{\theta_2}_s)]ds+M_t,$$
from which follows
$$ \frac{1}{t} \int_0^t (\varsigma_1 x^{\theta_1}_s+\varsigma_2 y^{\theta_2}_s)ds  \leq K+ \frac{M_t}{t}\cdot$$
Therefore, by \eqref{E19},
$$\limsup_{t\to \infty}  \frac{1}{t} \int_0^t (\varsigma_1 x^{\theta_1}_s+\varsigma_2 y^{\theta_2}_s)ds\leq K  \hspace{1cm} \text{ a.s.}$$
\end{proof}
\begin{remark}
For the deterministic version of model \eqref{E4}, i.e., $\sigma_i=\rho_i=0$ and the other coefficients are constants, it is easy to see that $\lim_{t\to \infty} y_t=0 $ holds under some special  conditions, i.e., the predator dies out, but it never get $\lim_{t\to \infty} x_t=0$ (if $\lim_{t\to \infty} y_t=0 $ then $\liminf_{t\to \infty} x_t\geq \frac{a_1}{b_1}>0$). However, in the above theorem, if $K=0$ then both prey and predator die out. This means that a relatively large stochastic perturbation can cause the extinction of the population. Further, the prey population dies out even if there is no predator and the death rate is so rapid (at an exponential rate). We can see that in two following theorems.
\end{remark}
\begin{theorem} \label{T6} 
Under condition {\rm(H1)}, if the prey is absent, i.e.,  $x_t=0$  a.s. for all $t\geq 0,$ then the predator dies with probability one. Furthermore, the death rate of predator is  exponential, i.e.,  
$$\limsup_{t\to \infty} \frac{\ln y_t}{t}\leq -\inf_{t\geq 0}\left [a_2(t)+\frac{\rho_1^2(t)}{2}\right]  \hspace{1cm} \text{ a.s.}$$
\end{theorem}
\begin{proof}
The quantity $y_t=\exp\{\eta_t\}$ of predator at the time $t$ satisfies the following equation 
$$d\eta_t=\left[-a_2-\frac{\rho_1^2}{2}-b_2 \exp\{\eta_t\} \right]dt+\rho_1dw_t.$$
Thus, 
\begin{align} \label{E28}
\eta_t=&\eta_0+ \int_0^t \left[-a_2-\frac{\rho_1^2}{2}-b_2 \exp\{\eta_s\} \right]ds+ M_t \notag\\
\leq & \eta_0 - \inf_{t\geq 0}\left [a_2(t)+\frac{\rho_1^2(t)}{2}\right]  t +M_t,
 \end{align}
 where 
 $M_t=\int_0^t \rho_1(s) dw_s$
 is a martingale. The quadratic variation of $M_t$
$$<M,M>_t=\int_0^t \rho_1^2(s)ds$$
satisfying
$$\limsup_{t\to \infty}\frac{<M,M>_t}{t}\leq {\rho_1^u}^2.$$
 Using the strong law of large numbers for martingales gives
 $\lim_{t\to \infty} \frac{M_t}{t}=0  \,\,  a.s.$
It then follows from \eqref{E28} that 
$$\limsup_{t\to \infty} \frac{\ln y_t}{t}\leq -\inf_{t\geq 0}\left [a_2(t)+\frac{\rho_1^2(t)}{2}\right]  $$
 and $\lim_{t\to \infty}y_t=0$  a.s. 
\end{proof}
\begin{theorem} \label{T7}
Under condition {\rm(H1)}, if the predator is absent, i.e., $y_t=0$ a.s. for all $t\geq 0,$ then  the quantity of prey satisfies the following
\begin{itemize}
\item [(i)] 
If $\sup_{t\geq 0}\left\{a_1(t)-\frac{\sigma_1^2(t)}{2}\right\}<0$ then $\lim_{t\to \infty} x_t=0 \text{ a.s.}$ and the prey dies out at an exponential rate;
\item[(ii)] If $\sup_{t\geq 0}\left\{a_1(t)-\frac{\sigma_1^2(t)}{2}\right\}=0$ then $\lim_{t\to \infty} \mathbb E x_t=0;$
\item [(iii)]  
$$\limsup_{t\to \infty} \frac{\ln x_t}{\ln t}\leq 1  \hspace{1cm} \text{ a.s.}$$
\end{itemize}
\end{theorem}
\begin{proof}
Similarly to Theorem \ref{T6}, the quantity $x_t=\exp\{\xi_t\}$ of prey at time $t$ satisfies the following equation
\begin{equation*}
d\xi_t=\left[a_1-\frac{\sigma_1^2}{2}-b_1 \exp\{\xi_t\} \right]dt+\sigma_1dw_t .
\end{equation*}
For Case (i), we have 
$$d\xi_t\leq \sup_{t\geq 0}\left\{a_1(t)-\frac{\sigma_1^2(t)}{2}\right\} dt+\sigma_1 dw_t.$$
Using the same arguments as in Theorem \ref{T6} yields
$$\limsup_{t\to \infty} \frac{\ln x_t}{t}\leq  \sup_{t\geq 0}\left\{a_1(t)-\frac{\sigma_1^2(t)}{2}\right\}<0,$$
 from which follows  that $\lim_{t\rightarrow \infty} x_t=0$ a.s. and the death rate of prey is exponential. Consider Case (ii). 
It follows from 
\begin{equation*}
\xi_t=\xi_0+\int_0^t \left[a_1-\frac{\sigma_1^2}{2}-b_1 \exp\{\xi_s\} \right]ds+\int_0^t \sigma_1(s) dw_s,
\end{equation*}
and Jensen's inequality that
$$\mathbb E\xi_t\leq \xi_0- b_1^l \int_0^t \mathbb E \exp\{\xi_s\}ds\leq \xi_0- b_1^l \int_0^t \exp\{\mathbb E\xi_s\}ds.$$
Therefore, $\mathbb E\xi_t \leq Z_t$ where $Z_t$ is the solution of the following differential equation
$$Z'_t=- b_1^l  \exp\{Z_t\}, \quad Z_0=\xi_0.$$
It is easy to see that 
$Z_t=-\log [b_1^l t+\exp\{-\xi_0\}] \to -\infty \text{ as } t\to \infty,$
then $\lim_{t\to \infty} \mathbb E\xi_t=-\infty.$ Using Jensen inequality again gives  $\mathbb E x_t=0.$
The proof of Case (iii) is similar to one of Theorem \ref{T4}. We therefore omit it here.

{\bf Acknowledgement.} 
The authors would like to thank the anonymous referee(s) for his very helpful suggestions which greatly improve the manuscript. 
\end{proof}

\end{document}